\documentclass[a4paper,fleqn, reqno]{amsart}
\usepackage[english]{babel}
\usepackage{graphicx}
\usepackage{xspace}
\usepackage{hyperref}
\usepackage{geometry}
\usepackage{color}
\usepackage{amsmath, amssymb, amscd}
\usepackage{booktabs}
\usepackage{amsfonts}
\usepackage{amsthm}
\usepackage{latexsym}
\usepackage{wrapfig}
\usepackage{float}
\usepackage{subfig}
\usepackage{mathtools}
\newtheorem{mythe}{Theorem}[section]
\newtheorem{mylem}[mythe]{Lemma}
\newtheorem{mycon}[mythe]{Conjecture}

\newtheorem{myass}[mythe]{Assumption}

\newcommand{\tr}{\mathrm{Tr}}

\newcommand{\modu}{\ \mathrm{mod}\ }

\newcommand{\sym}{\mathrm{Sym}}

\newcommand{\widesim}[2][1.5]{\mathrel{\overset{#2}{\scalebox{#1}[1]{$\sim$}}}}

\begin{document}
\title{The variance of the Euler totient function}
\author{Tom van Overbeeke}
\address{Mathematisch Instituut\\ Universiteit Utrecht\\ Postbus 80.010\\ 3508 TA Utrecht, Nederland}
\email{tomvov@gmail.com}
\thanks{The author would like to thank Gunther Cornelissen for all the help writing this paper and discussing its contents. He would also like to thank Zeev Rudnick for his input and discussions concerning the subject.}

\begin{abstract}
In this paper we study the variance of the Euler totient function (normalized to $\varphi(n)/n$) in the integers $\mathbb{Z}$ and in the polynomial ring $\mathbb{F}_q[T]$ over a finite field $\mathbb{F}_q$. It turns out that in $\mathbb{Z}$, under some assumptions, the variance of the normalized Euler function becomes constant. This is supported by several numerical simulations. Surprisingly, in $\mathbb{F}_q[T]$, $q\rightarrow \infty$, the analogue does not hold: due to a high amount of cancellation, the variance becomes inversely proportional to the size of the interval.
\end{abstract}
\maketitle
\section{Introduction} \label{pa:intr}

There are many results and conjectures about the statistical behaviour of arithmetical functions in short intervals. A few examples are the cancellation of the M\"obius function $\sum_{n<x}\mu(x)$, discussed in any book on analytic number theory, see e.g. \cite{Apo76}, or the conjecture about the variance of the von Mangoldt function $\Lambda$, due to Goldstein and Montgomery \cite{Gol87}. Recently, Keating and Rudnick \cite{Kea16, Kea14} have developed an entirely new technique for studying such problems in polynomial rings over finite fields, using random matrix theory, and it has been applied successfully to prove the analogue of some of these conjectures (e.g., for the von Mangoldt function \cite{Kea14}, the M\"obius function \cite{Kea16} and the divisor function \cite{Rud16}). 

In this paper, we study the analogue for the variance of the Euler totient function. In section \ref{ch:Z}, we review some known results, mainly due to Chowla \cite{Cho32} and Montgomery \cite{Mon87}. The average of the normalized function $\varphi(n)/n$ is  given by
\begin{equation} \label{av1} \frac{1}{H} \sum_{X<n <X+H} \frac{\varphi(n)}{n} \sim \frac{1}{\zeta(2)}. \end{equation}  We will first propose a conjecture (\ref{con}), based on numerical work, for the variance of $\varphi(n)/n$ in short intervals, namely, 
\begin{equation} \label{c1} \frac{1}{X} \sum_{x<X} \left( \sum_{x<n <x+H} \frac{\varphi(n)}{n} - \frac{H}{\zeta(2)} \right)^2 \widesim{?} \frac{1}{6 \zeta(2)} - \frac{1}{6 \zeta(2)^2} , \end{equation}
where $H=\Theta(x^\delta)$, $0<\delta\leq 1$.
Notice that the conjectured variance doesn't depend on the length of the interval. 

 In Theorem \ref{th:Hx}, we prove a partial result in this direction for intervals of the form $[x,2x]$ (i.e., $H=x$), namely, we prove that a related limit is equal to the right hand side of Equation (\ref{c1}), so that the conjecture in this case becomes equivalent to a problem of interchanging two limits. We also study the case where $H=[x^\delta]$ for $0<\delta<1$. Again, in Theorem \ref{th:Hxdelta} we prove a formula for a related limit, albeit under the assumption of uncorrelatedness of $[x^\delta]$ and $x$ modulo integers, cf.\ \ref{ass}. 

In the second part of the paper, we study the analogue of these problems for the polynomial ring $\mathbb{F}_q[T]$. Here, we essentially use the method of Keating and Rudnick to obtain some definite results. The Euler totient function of $f \in \mathbb{F}_q[T]$ is given by $\varphi(f):=\#\left(\mathbb{F}_q[T]/(f)\right)^\times$. 
Define the norm of $f$ to be 
$||f||:=\#\left(\mathbb{F}_q[T]/(f)\right)=q^{\deg(f)},$ and denote 
$$\mathcal{M}_n:=\{f\in \mathbb{F}_q[T]\ | \ f \textrm{ monic of degree $n$}\},$$ a set of size $q^n$. For $0\leq h \leq n-2$ and $A \in \mathcal{M}_n,$ define a \emph{short interval of size $q^h$ around $A$}  by
$$I(A;h):=\{f\in \mathbb{F}_q[T] \ | \ ||f-A||\leq q^h\}.$$
In complete analogy to Equation (\ref{av1}), the average of the normalized totient function in   $\mathbb{F}_q[T]$ is 
\begin{equation} \label{av2} \frac{1}{q^n}\sum_{A\in\mathcal{M}_n}\frac{1}{q^h}\sum_{f\in I(A;h)}\frac{\varphi(f)}{||f||} = \frac{1}{\zeta_q(2)}, \end{equation}
for all $n \geq 1$, where $\zeta_q(s) = (1-q^{1-s})^{-1}$ is the zeta function of $ \mathbb{F}_q[T]$. 

In our main theorem \ref{th:TH1}, we prove that for fixed $n$ and $0\leq h\leq n-5$, the variance is given by 
\begin{equation} \label{f}  \frac{1}{q^n}\sum_{A\in\mathcal{M}_n}\left|\sum_{f\in I(A;h)}\frac{\varphi(f)}{||f||}- \frac{q^h}{\zeta_q(2)}\right|^2 
\widesim{q \rightarrow \infty} q^{-h-3}. \end{equation} 
Somewhat surprisingly, the variance is inversely proportional to the length of the interval. Thus, the result is very different from the expected value in the integers (\ref{c1}). We do not have a conceptual explanation for this, except the (unexpected?) cancellation of terms in the proof of the result.  It would be interesting to know whether this phenomenon persists for other interesting arithmetical functions in polynomial rings and more general function fields.

\section{The ring of integers}\label{ch:Z}
Let us first consider the normalized Euler totient function in the ring of integers. Many introductory textbooks in number theory prove that
$$\sum_{n \leq X}\frac{\varphi(n)}{n}\sim \frac{X}{\zeta(2)}.$$
Define the remainder term by
$$R_0(x):=\sum_{n\leq x}\frac{\varphi(n)}{n}-\frac{x}{\zeta(2)}.$$
A few things are known about this function. Defining the fractional part function\\ $\{x\}=x-[x]$ for any real $x$ and applying the known equality $\sum_{d|n}\frac{\mu(d)}{d}=\frac{\varphi(n)}{n}$, it is easy to prove that we can write this remainder term as an infinite sum:
\begin{equation}\label{eq:rem}
R_0(x)=-\sum_{n=1}^\infty \frac{\mu(n)}{n}\bigg\{\frac{x}{n}\bigg\}.
\end{equation}
The size of this remainder term, as well as the size of the related term $R(x):=\sum_{n\leq x}\varphi(n) - \frac{x^2}{2\zeta(2)}$, has been studied by several people. Montgomery \cite{Mon87} showed that 
$$R_0(x)=\frac{R(x)}{x}+O\left(\exp(-c\sqrt{\log x})\right).$$
Walfisz \cite{Wal63} improved earlier work of Mertens \cite{Mer77} by showing that $$R(x)=o\left( x(\log x)^{\frac{2}{3}}(\log\log x)^{\frac{4}{3}}\right).$$
Together this implies that
$$R_0(x)=o\left( (\log x)^{\frac{2}{3}}(\log\log x)^{\frac{4}{3}}\right).$$
Finally Montgomery \cite{Mon87} also showed that $$R_0(x)=\Omega_\pm ((\log \log x)^{\frac{1}{2}}).$$ 
As to the averages of this remainder term, it is easy to show that the continuous average tends to zero
$$\frac{1}{X}\int_1^XR_0(t)dt\sim 0,$$
while the discrete average tends to $\frac{1}{2\zeta(2)}$,
$$\frac{1}{X}\sum_{x\leq X}R_0(x)\sim \frac{1}{2\zeta(2)}.$$
The continuous mean square was first calculated by Chowla \cite{Cho32}, who showed that
$$\frac{1}{X}\int_1^XR_0(t)^2dt\sim \frac{1}{12\zeta(2)}.$$
Erd\"os and Shapiro \cite{Sha55} noted that this continuous mean square implies the discrete mean square to be given by
$$\frac{1}{X}\sum_{x\leq X}R_0(x)^2 \sim \frac{1}{12\zeta(2)}+\frac{1}{6\zeta(2)^2}.$$
Continuing on the work of Chowla, Erd\"os and Shapiro it might be possible to prove equation (\ref{c1}) directly.\\
To calculate the variance of the (normalized) Euler totient function in an interval of size $H$, we define the remainder term for an interval
$$R_0(x;H):=R_0(x+H)-R_0(x)=\sum_{x<n\leq x+H}\frac{\varphi(n)}{n}-\frac{H}{\zeta(2)}.$$
In this paper we will consider the discrete squared average of this function, both for $H=x$ and for $H=[x^\delta]$ (short intervals). We conjecture that
\begin{mycon}\label{con}
Let $H=\Theta(x^\delta)$, for some fixed $0<\delta\leq 1$, be the size of the interval. Then
$$\frac{1}{X}\sum_{x\leq X}R_0(x,H)^2\sim \frac{1}{6\zeta(2)}-\frac{1}{6\zeta(2)^2}.$$
\end{mycon}

\begin{figure}
    \centering
    \includegraphics[width=1\textwidth]{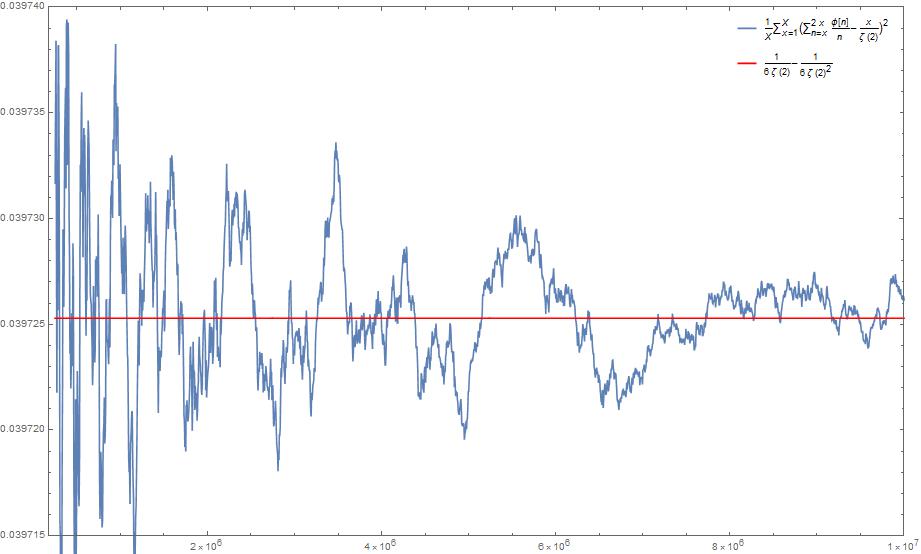}
    \caption{$\frac{1}{X}\sum_{x=1}^X\left(\sum_{n=x+1}^{2x}\frac{\varphi(n)}{n}-\frac{x}{\zeta(2)}\right)^2$ for $2.5 \cdot 10^5\leq X \leq 10^7$.}
    \label{fig:varsum}
\end{figure}

Substituting equation (\ref{eq:rem}) into $\frac{1}{X}\sum_{x\leq X}R_0(x,H)^2$, we see that we are interested in the expression
$$\lim_{X\rightarrow \infty}\frac{1}{X}\sum_{x\leq X}\left(\sum_{n=1}^\infty \frac{\mu(n)}{n}\bigg\{\frac{x+H}{n}\bigg\}-\sum_{n=1}^\infty \frac{\mu(n)}{n}\bigg\{\frac{x}{n}\bigg\}\right)^2$$
$$=\lim_{X\rightarrow \infty}\frac{1}{X}\sum_{x=1}^X\sum_{m,n=1}^\infty \frac{\mu(m)\mu(n)}{mn}\left(\bigg\{\frac{x}{n}\bigg|\frac{x}{m}\bigg\}-\bigg\{\frac{x+H}{n}\bigg|\frac{x}{m}\bigg\}-\bigg\{\frac{x}{n}\bigg|\frac{x+H}{m}\bigg\}+\bigg\{\frac{x+H}{n}\bigg|\frac{x+H}{m}\bigg\}\right),$$
where we have introduced the notation $\{\frac{x}{n}|\frac{y}{m}\}:=\{\frac{x}{n}\}\{\frac{y}{m}\}$ to shorten our expressions a bit.
In this paper we prove two partial results to showing that this expression equals $\frac{1}{6\zeta(2)}-\frac{1}{6\zeta(2)^2}$. The first one is for the case $H=x$.
\begin{mythe}\label{th:Hx}
$$
\sum_{m,n=1}^\infty\lim_{X\rightarrow \infty}\frac{1}{X}\sum_{x=1}^X \frac{\mu(m)\mu(n)}{mn}\left(\bigg\{\frac{x}{n}\bigg|\frac{x}{m}\bigg\}-\bigg\{\frac{2x}{n}\bigg|\frac{x}{m}\bigg\}-\bigg\{\frac{x}{n}\bigg|\frac{2x}{m}\bigg\}+\bigg\{\frac{2x}{n}\bigg|\frac{2x}{m}\bigg\}\right)=\frac{1}{6\zeta(2)}-\frac{1}{6\zeta(2)^2}
$$
\end{mythe}
The second theorem is for the case $H=[x^\delta]$, $0<\delta<1$.
\begin{mythe}\label{th:Hxdelta}
Fix $0<\delta<1$. Assuming \ref{ass}, we have
$$
\sum_{m,n=1}^\infty\lim_{X\rightarrow \infty}\frac{1}{X}\sum_{x=1}^X \frac{\mu(m)\mu(n)}{mn}\left(\bigg\{\frac{x}{n}\bigg|\frac{x}{m}\bigg\}-\bigg\{\frac{x+[x^\delta]}{n}\bigg|\frac{x}{m}\bigg\}-\bigg\{\frac{x}{n}\bigg|\frac{x+[x^\delta]}{m}\bigg\}+\bigg\{\frac{x+[x^\delta]}{n}\bigg|\frac{x+[x^\delta]}{m}\bigg\}\right)$$
$$=\frac{1}{6\zeta(2)}-\frac{1}{6\zeta(2)^2}
$$
\end{mythe}
We prove these theorems in sections \ref{pa:Hx} and \ref{pa:Hxdelta} respectively. Note that the expressions in these theorems are the same as the one we are interested in, but with the infinite summation over $m,n$ and the summation of $X$ interchanged. It is not clear at all that interchanging these summations does not change the value of this expression. We have not been able to prove this, due to the lack of absolute convergence of these summations. Numerical simulations however suggest that this is indeed the case. In figure \ref{fig:varsum} the variance is given up to $X=10^7$ for $H=x$. In figure \ref{fig:varsumdelta} the variance is shown for $H=[x^\delta]$, $\delta=\frac{1}{2}$, $\delta=\frac{5}{6}$.\\
\begin{figure}
    \centering
    \includegraphics[width=1\textwidth]{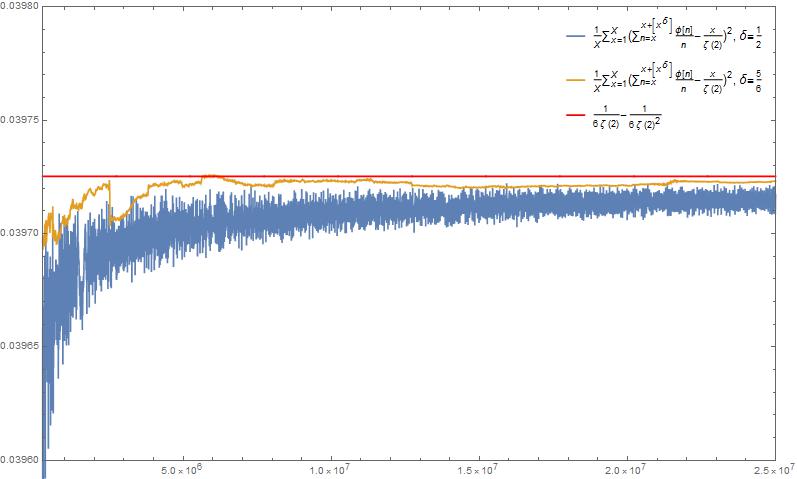}
    \caption{$\frac{1}{X}\sum_{x=1}^X\left(\sum_{n=x}^{x+2[x^\delta]}\frac{\varphi(n)}{n}-\frac{x}{\zeta(2)}\right)^2$, where $X$ ranges from 0 to $2.5 \cdot 10^7$ and $\delta=\frac{1}{2},\frac{5}{6}$; together with the predicted value for the variance.}
    \label{fig:varsumdelta}
\end{figure}
\\
Finally note that Erd\"os' and Shapiro's result for the mean square suggests that the discrete variance equals $\frac{1}{6\zeta(2)}+\frac{1}{3\zeta(2)^2}$ \emph{if} $G(x+H)$ and $G(x)$ were independent. Conjecture \ref{con} implies that this is not the case. That is, $G(x+H)$ and $G(x)$ are not independent. This is proven in section \ref{pa:evod}, as an application of Assumption \ref{ass}. We calculate the expression, again with the infinite summations interchanged, for $H=2[x^\delta]$ and $H=2[x^\delta]+1$ and show that this tends to $\frac{1}{6\zeta(2)}-\frac{2}{9\zeta(2)^2}$ and $\frac{1}{6\zeta(2)}-\frac{1}{9\zeta(2)^2}$ in the respective cases. Numerical simulations suggest the same distinction, as shown in figure \ref{fig:varsum2delta} for $\delta=\frac{1}{4}, \frac{1}{2},\frac{5}{6}$.
\begin{figure}
    \centering
    \includegraphics[width=1\textwidth]{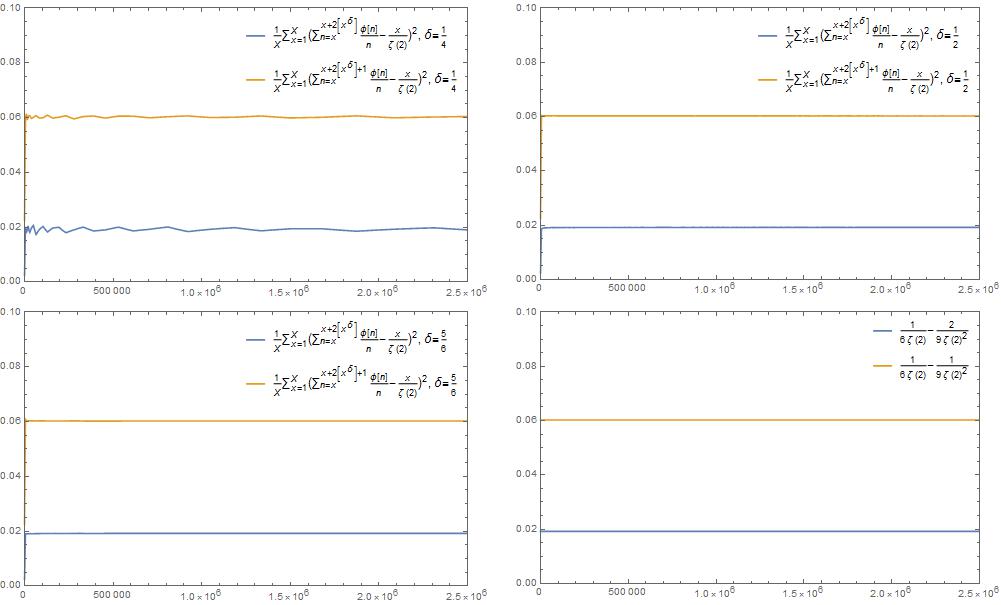}
    \caption{\ \\$\frac{1}{X}\sum_{x=1}^X\left(\sum_{n=x}^{x+2[x^\delta]}\frac{\varphi(n)}{n}-\frac{x}{\zeta(2)}\right)^2$ and $\frac{1}{X}\sum_{x=1}^X\left(\sum_{n=x}^{x+2[x^\delta]+1}\frac{\varphi(n)}{n}-\frac{x}{\zeta(2)}\right)^2$, where $X$ ranges from 0 to $2.5 \cdot 10^6$ and $\delta=\frac{1}{4},\frac{1}{2},\frac{5}{6}$; together with the predicted values for the respective variance.}
    \label{fig:varsum2delta}
\end{figure}

\section{Proof of Theorem \ref{th:Hx}}\label{pa:Hx}
In this section we prove Theorem \ref{th:Hx}. To do so we need to prove a couple of lemmas. We first give the proof of the theorem to give an idea how we will apply these lemmas.

\begin{mythe}\label{th:varsumphi}[Restatement of Theorem \ref{th:Hx}.]
$$
\sum_{m,n=1}^\infty\lim_{X\rightarrow \infty}\frac{1}{X}\sum_{x=1}^X \frac{\mu(m)\mu(n)}{mn}\left(\bigg\{\frac{x}{n}\bigg|\frac{x}{m}\bigg\}-\bigg\{\frac{2x}{n}\bigg|\frac{x}{m}\bigg\}-\bigg\{\frac{x}{n}\bigg|\frac{2x}{m}\bigg\}+\bigg\{\frac{2x}{n}\bigg|\frac{2x}{m}\bigg\}\right)=\frac{1}{6\zeta(2)}-\frac{1}{6\zeta(2)^2}
$$
\end{mythe}
\begin{proof}
Note that if $\sum_{k=1}^{mn} f(x+k,m,n)=g(m,n)$ for all $x\in \mathbb{Z}_{\geq 0}$ for some functions $f,g$, then $\lim_{X\rightarrow\infty}\frac{1}{X}\sum_{x=1}^X f(x,m,n)=\frac{1}{mn}g(m,n)$. 
We apply this several times for $f(x,m,n)$ equal to $\{\frac{x}{n}|\frac{x}{m}\}$, $\{\frac{2x}{n}|\frac{x}{m}\}$, $\{\frac{x}{n}|\frac{2x}{m}\}$, $\{\frac{2x}{n}|\frac{2x}{m}\}$ respectively. Lemma \ref{le:sumymn} states the different $g(m,n)$ for these respective $f(x,m,n)$. The resulting expression, given by
\begin{align}
&\sum_{\substack{m\ \mathrm{odd},\\ n\ \mathrm{odd}}}\frac{\mu(m)\mu(n)}{m^2n^2}\left(\frac{\gcd(m,n)^2-1}{12}\right)-
\sum_{\substack{m\ \mathrm{even},\\ n\ \mathrm{odd}}}\frac{\mu(m)\mu(n)}{m^2n^2}\left(\frac{\gcd(m,n)^2-1}{24}\right) \nonumber\\
& \ \ \ \ \ \ \ \ \ -\sum_{\substack{m\ \mathrm{odd},\\ n\ \mathrm{even}}}\frac{\mu(m)\mu(n)}{m^2n^2}\left(\frac{\gcd(m,n)^2-1}{24}\right)+\sum_{\substack{m\ \mathrm{odd},\\ n\ \mathrm{odd}}}\frac{\mu(m)\mu(n)}{m^2n^2}\left(\frac{\gcd(m,n)^2+2}{12}\right), \label{eq:odev}
\end{align}
we calculate using Lemmas \ref{le:sumneo}, \ref{le:summneo}. The theorem then follows immediately.
\end{proof} 
\begin{mylem}\label{le:sumymn}
Let $m, n$ be positive integers. Fix $x\in \mathbb{Z}_{\geq 0}$. Then
\begin{enumerate}
\item
$$\sum_{k=1}^{mn}\bigg\{\frac{x+k}{m}\bigg|\frac{x+k}{n}\bigg\}=\frac{(m-1)(n-1)}{4}+\frac{\gcd(m,n)^2-1}{12}.$$
\item
$$\sum_{k=1}^{mn}\bigg\{\frac{2(x+k)}{m}\bigg|\frac{x+k}{n}\bigg\}=\left\{ \begin{array}{l l} \frac{(m-2)(n-1)}{4}+\frac{\gcd(\frac{m}{2},n)^2-1}{6} & \textrm{if $m$ is even,}\\
\frac{(m-1)(n-1)}{4}+\frac{\gcd(m,n)^2-1}{24}& \textrm{if $m$ is odd.}
\end{array}\right.
$$
\item
$$\sum_{k=1}^{mn}\bigg\{\frac{2(x+k)}{m}\bigg|\frac{2(x+k)}{n}\bigg\}
=\left\{ \begin{array}{l l} \frac{(m-2)(n-2)}{4}+\frac{\gcd(m,n)^2-4}{12}& \textrm{if $m$ is even, $n$ is even,}\\
\frac{(m-2)(n-1)}{4}+\frac{\gcd(m,n)^2-1}{12}& \textrm{if $m$ is even, $n$ is odd,}\\
\frac{(m-1)(n-2)}{4}+\frac{\gcd(m,n)^2-1}{12}& \textrm{if $m$ is odd, $n$ is even,}\\
\frac{(m-1)(n-1)}{4}+\frac{\gcd(m,n)^2-1}{12}& \textrm{if $m$ is odd, $n$ is odd.}
\end{array}\right.
$$
\end{enumerate}
\end{mylem}
\begin{proof}
Note that if $x\equiv r_m \modu m$, then $\left\{\frac{x}{m}\right\}=\frac{r_m}{m}$. It follows that $\left\{\frac{x+mn}{m}\big|\frac{x+mn}{n}\right\}=\left\{\frac{x}{m}\big|\frac{x}{n}\right\}$. It is hence sufficient to prove the respective statements for $x=0$. Write $d=\gcd(m,n)$. 
\begin{enumerate}
\item Since $$\sum_{k=1}^{mn}\bigg\{\frac{k}{m}\bigg|\frac{k}{n}\bigg\}=d\sum_{k=1}^{\frac{mn}{d}}\bigg\{\frac{k}{m}\bigg|\frac{k}{n}\bigg\},$$ the crucial question is which pairs $(k\modu m, k\modu n)$ are attained when $k$ runs over the integers from 1 to $\frac{mn}{d}$. Applying some combinatorial arguments, it is easy to see that these are exactly the pairs $(r_m,r_n)$, such that $r_m\equiv r_n \modu d$. Hence
\begin{align*}
\sum_{k=1}^{mn}\bigg\{\frac{k}{m}\bigg|\frac{k}{n}\bigg\}&=d\sum_{k=1}^{\frac{mn}{d}}\bigg\{\frac{k}{m}\bigg|\frac{k}{n}\bigg\}=d\sum_{l=0}^{d-1} \left(\sum_{\substack{0\leq r_m<m\\ r_m \equiv l \modu d}}\frac{r_m}{m}\right)\left(\sum_{\substack{0\leq r_n<n\\ r_n \equiv l \modu d}}\frac{r_n}{n}\right)\\
&=\frac{(m-1)(n-1)}{4}+\frac{d^2-1}{12}.
\end{align*}
\item Note that if $m$ is even, we can apply $(1)$ to $\frac{m}{2}$. If $m$ is odd, the proof is analogous to $(1)$. It turns out that the pairs $(2k\modu m, k\modu n)$ that are attained when $k$ runs over 1 to $\frac{mn}{d}$ are exactly the pairs $(r_m,r_n)$ such that $r_m \equiv 2r_n\mod d$. The calculations are then analogous to the calculations for proving $(1)$.\\
\item For the last statement, we can apply $(1)$ to $\frac{m}{2},\frac{n}{2}$ if both $m$ and $n$ are even. If $m$ is odd, $n$ is even, we apply Lemma $(2)$ on $m,\frac{n}{2}$ and vice versa for $m$ even, $n$ odd. Finally if both $m$ and $n$ are odd, then $(2x \modu m, 2x\modu n)$ runs over the same pairs as $(x\modu m, x\modu n)$, only in a different order. We can hence apply $(1)$.
\end{enumerate}
\end{proof}
\begin{mylem}[\cite{Cho32}, Lemma 6] \label{le:sumgcd}
\begin{equation}\label{eq:gcd}
\sum_{m,n=1}^\infty\frac{\mu(m)\mu(n)}{m^2n^2}\cdot \gcd(m,n)^2=\frac{1}{\zeta(2)}.
\end{equation}
\qed
\end{mylem}
\noindent Equation (\ref{eq:odev}) suggests that we specifically need to know the values of the even and odd sums. We have
\begin{mylem}\label{le:sumneo}\ 
\begin{enumerate}
\item
$$\sum_{n\ \mathrm{even}}\frac{\mu(n)}{n^2}=-\frac{1}{3\zeta(2)}$$
\item $$\sum_{n\ \mathrm{odd}}\frac{\mu(n)}{n^2}=\frac{4}{3\zeta(2)}.$$
\end{enumerate}
\end{mylem}
\begin{proof}
Both parts follow directly from
$$\sum_{n\ \mathrm{even}}\frac{\mu(n)}{n^2}=\sum_{n}\frac{\mu(2n)}{(2n)^2} =\sum_{n\ \mathrm{odd}}\frac{\mu(2n)}{(2n)^2}=-\frac{1}{4}\sum_{n\ \mathrm{odd}}\frac{\mu(n)}{n^2}.$$
\end{proof}
\noindent Analogously, applying Lemma \ref{le:sumgcd}, we find 
\begin{mylem}\label{le:summneo}\ 
\begin{enumerate}
\item $$\sum_{\substack{m\ \mathrm{odd},\\ n\ \mathrm{odd}}}\frac{\mu(m)\mu(n)}{m^2n^2}\gcd(m,n)^2=\frac{4}{3\zeta(2)}.$$
\item $$\sum_{\substack{m\ \mathrm{even},\\ n\ \mathrm{odd}}}\frac{\mu(m)\mu(n)}{m^2n^2}\gcd(m,n)^2=\sum_{\substack{m\ \mathrm{odd},\\ n\ \mathrm{even}}}\frac{\mu(m)\mu(n)}{m^2n^2}\gcd(m,n)^2=-\frac{1}{3\zeta(2)}.$$
\item $$\sum_{\substack{m\ \mathrm{even},\\ n\ \mathrm{even}}}\frac{\mu(m)\mu(n)}{m^2n^2}\gcd(m,n)^2=\frac{1}{3\zeta(2)}.$$
\end{enumerate}
\qed
\end{mylem}
\section{Proof of Theorem \ref{th:Hxdelta}}\label{pa:Hxdelta}
\noindent In short intervals the calculations are more difficult. For any $x,H$, knowing the value of $\left\{\frac{x+H}{m}\right\}$ is equivalent to knowing $x+H \mod m$. For $H=[x^\delta]$ this is difficult, as the value of $[x^\delta] \modu m$ is independent to the value of $x\modu n$ for all $m,n$. This is easily seen. For example if $\delta=\frac{1}{2}$, then $y=[x^\delta]$ implies $y^2\leq x < (y+1)^2$. Since the gaps between $y^2$ and $(y+1)^2$ get larger and larger, at some point the gaps will be much bigger than $m$. Hence if we fix some large $y\in\mathbb{Z}$ and let $x\in \mathbb{Z}$ range in between $y^2$ and $(y+1)^2$, we will find every possible value for $x\modu m$ with about the same probability. This qualitative argument persuades us to make the following assumption.
\begin{myass}\label{ass}
Fix $m,n\in \mathbb{Z}_{\geq 1}$. For any $0<\delta<1$, there exists no correlation between $x\modu n$ and $[x^\delta] \modu m$. That is
$$\mathbb{P}\left([x^\delta] \equiv r_m \modu m | x\equiv r_n \modu n\right)=\frac{1}{m}$$
for any $0\leq r_m<m$, $0\leq r_n< n$.
\end{myass}
\noindent This assumption enables us to predict the average of the desired functions, even though we do not know $[x^\delta] \modu m$.
\begin{mylem}\label{le:xxdelta}
Assuming \ref{ass}, we have for any $m,n$
$$\lim_{X\rightarrow \infty} \frac{1}{X}\sum_{x=1}^X \bigg\{\frac{x}{m}\bigg|\frac{x+[x^\delta]}{n}\bigg\}=\frac{1}{mn}\left(\frac{(m-1)(n-1)}{4}\right).$$
\end{mylem}
\begin{proof}
By Assumption \ref{ass} there is no correlation between $x \modu m$ and $[x^\delta]\modu n$. Hence any pair $(x\modu m, x+[x^\delta]\modu n)$ is attained with equal probability. Since there are $mn$ different such pairs, we conclude that the average is equal to
$$
\frac{1}{mn}\left(\sum_{r_m=0}^{m-1}\frac{r_m}{m}\right)\left(\sum_{r_n=0}^{n-1}\frac{r_n}{n}\right)=\frac{1}{mn}\left(\frac{(m-1)(n-1)}{4}\right).
$$
\end{proof}
\begin{mythe}[Restatement of Theorem \ref{th:Hxdelta}]
Fix $0<\delta<1$. Assuming \ref{ass}, we have
$$
\sum_{m,n=1}^\infty\lim_{X\rightarrow \infty}\frac{1}{X}\sum_{x=1}^X \frac{\mu(m)\mu(n)}{mn}\left(\bigg\{\frac{x}{n}\bigg|\frac{x}{m}\bigg\}-\bigg\{\frac{x+[x^\delta]}{n}\bigg|\frac{x}{m}\bigg\}-\bigg\{\frac{x}{n}\bigg|\frac{x+[x^\delta]}{m}\bigg\}+\bigg\{\frac{x+[x^\delta]}{n}\bigg|\frac{x+[x^\delta]}{m}\bigg\}\right)$$
$$=\frac{1}{6\zeta(2)}-\frac{1}{6\zeta(2)^2}
$$
\end{mythe}
\begin{proof}
The proof of this is analogous to the proof of Theorem  \ref{th:varsumphi}. Note that by Assumption \ref{ass} and Lemma \ref{le:sumymn}
\begin{align*}
\lim_{X\rightarrow \infty}\frac{1}{X}\sum_{x=1}^X \bigg\{\frac{x+[x^\delta]}{n}\bigg\}\bigg\{\frac{x+[x^\delta]}{m}\bigg\}
&=\lim_{X\rightarrow \infty}\frac{1}{X}\sum_{x=1}^X\bigg\{\frac{x}{n}\bigg\}\bigg\{\frac{x}{m}\bigg\}\\
&=\frac{1}{mn}\left(\frac{(m-1)(n-1)}{4}+\frac{\gcd(m,n)^2-1}{12}\right).
\end{align*}
\end{proof}
\section{Fixing the parity of $H$}\label{pa:evod}
As noted in the introduction the variance of $\varphi(n)/n$ seems to be much lower if we force $H$ to be even, than if we force $H$ to be odd. In this section we show that this is indeed the case by calculating the variance for $H=2[x^\delta]$ and $H=2[x^\delta]+1$, assuming \ref{ass}. 
\begin{mylem}\label{le:x2xdelta}
Assuming \ref{ass}, we have for any $m,n$
$$\lim_{X\rightarrow \infty} \frac{1}{X}\sum_{x=1}^X \bigg\{\frac{x}{m}\bigg|\frac{x+2[x^\delta]}{n}\bigg\}=\left\{ \begin{array}{l l} \frac{1}{mn}\left(\frac{(m-1)(n-1)}{4}+\frac{1}{4}\right) & \textrm{if $m$ is even, $n$ is even}\\
\frac{1}{mn}\left(\frac{(m-1)(n-1)}{4}\right)& \textrm{otherwise.}
\end{array}\right.$$
\end{mylem}
\begin{proof}
If either $m$ or $n$ is odd, the statement follows using some analogous arguments as in the proof of Lemma \ref{le:xxdelta}.
If both $m$ and $n$ are even, then $x\modu m$ and $x+2[x^\delta]\modu n$ have the same parity. Since there are $\frac{mn}{2}$ pairs $(r_m \modu m, r_n \modu n)$ with the same parity and each of these pairs is equally probable, we conclude that
\begin{align*}
\lim_{X\rightarrow \infty} \frac{1}{X}\sum_{x=1}^X \bigg\{\frac{x}{m}\bigg|\frac{x+2[x^\delta]}{n}\bigg\}
&=\frac{2}{mn}\sum_{l\in \{0,1\}}\left(\sum_{\substack{0\leq r_m<m\\ r_m \equiv l \modu 2}}\frac{r_m}{m}\right)\left(\sum_{\substack{0\leq r_n<n\\ r_n \equiv l \modu 2}}\frac{r_n}{n}\right)\\
&=\frac{1}{mn}\left(\frac{(m-1)(n-1)}{4}+\frac{1}{4}\right).
\end{align*}
\end{proof}
\begin{mylem}\label{le:x2xdelta1}
Assuming \ref{ass}, we have for any $m,n$
$$\lim_{X\rightarrow \infty} \frac{1}{X}\sum_{x=1}^X \bigg\{\frac{x}{m}\bigg|\frac{x+2[x^\delta]+1}{n}\bigg\}=\left\{ \begin{array}{l l} \frac{1}{mn}\left(\frac{(m-1)(n-1)}{4}-\frac{1}{4}\right) & \textrm{if $m$ is even, $n$ is even}\\
\frac{1}{mn}\left(\frac{(m-1)(n-1)}{4}\right)& \textrm{otherwise.}
\end{array}\right.$$
\end{mylem}
\begin{proof}
This is exactly analogous to the proof of Lemma \ref{le:x2xdelta}
\end{proof}
\begin{mythe}\label{th:x2deltax}
Assuming \ref{ass}, we have
$$
\sum_{m,n=1}^\infty\lim_{X\rightarrow \infty}\frac{1}{X}\sum_{x=1}^X \frac{\mu(m)\mu(n)}{mn}\left(\bigg\{\frac{x}{n}\bigg|\frac{x}{m}\bigg\}-\bigg\{\frac{x+2[x^\delta]}{n}\bigg|\frac{x}{m}\bigg\}-\bigg\{\frac{x}{n}\bigg|\frac{x+2[x^\delta]}{m}\bigg\}\right. $$ $$ \left.+\bigg\{\frac{x+2[x^\delta]}{n}\bigg|\frac{x+2[x^\delta]}{m}\bigg\}\right)
=\frac{1}{6\zeta(2)}-\frac{2}{9\zeta(2)^2}.
$$
\end{mythe}
\begin{proof}
Applying Lemma \ref{le:sumymn} and Lemma \ref{le:x2xdelta} in the same way as in the proof of Theorem \ref{th:Hxdelta}, we see that the value we want to calculate equals
$$\sum_{m,n=1}^\infty \frac{\mu(m)\mu(n)}{m^2n^2}\left(\frac{\gcd(m,n)^2-1}{6}\right)-\frac{1}{2}\sum_{\substack{m\ \mathrm{even},\\ n\ \mathrm{even}}}\frac{\mu(m)\mu(n)}{m^2n^2}=\frac{1}{6\zeta(2)}-\frac{2}{9\zeta(2)^2}.$$
\end{proof}
\begin{mythe}
Assuming \ref{ass}, we have
$$
\sum_{m,n=1}^\infty\lim_{X\rightarrow \infty}\frac{1}{X}\sum_{x=1}^X \frac{\mu(m)\mu(n)}{mn}\left(\bigg\{\frac{x}{n}\bigg|\frac{x}{m}\bigg\}-\bigg\{\frac{x+2[x^\delta]+1}{n}\bigg|\frac{x}{m}\bigg\}-\bigg\{\frac{x}{n}\bigg|\frac{x+2[x^\delta]+1}{m}\bigg\}\right. $$ $$ \left.+\bigg\{\frac{x+2[x^\delta]+1}{n}\bigg|\frac{x+2[x^\delta]+1}{m}\bigg\}\right)
=\frac{1}{6\zeta(2)}-\frac{1}{9\zeta(2)^2}.
$$
\end{mythe}
\begin{proof}
Analogous to the proof of Theorem \ref{th:x2deltax}.
\end{proof}
\section{The polynomial ring over $\mathbb{F}_q$}\label{pa:fqt}
In this section we use the same notation as in the introduction. Furthermore we denote the normalized Euler totient function as $\beta(f):=\frac{\varphi(f)}{||f||}$.  For $A\in \mathcal{M}_n$, $0\leq h \leq n-2$, we write
$$\mathcal{N}_\beta(A;h):=\sum_{f\in I(A;h)}\beta(f).$$
The variance, (for fixed $n$ and $0\leq h\leq n-2$), is defined as
$$\textrm{Var}(\mathcal{N}_\beta):=\frac{1}{q^n}\sum_{A\in\mathcal{M}_n}\left|\mathcal{N}_\beta(A;h)-\frac{1}{q^n}\sum_{A\in\mathcal{M}_n}\mathcal{N}_\beta(A;h)\right|^2.$$
The main theorem of this paper, as was introduced in section \ref{pa:intr}, is then given by
\begin{mythe}\label{th:TH1}
Fix $0\leq h \leq n-5$. As $q\rightarrow \infty$,
$$\mathrm{Var}( \mathcal{N}_{\beta})\sim q^{-h-3}.$$
\end{mythe}
For the proof, we first introduce the notions of Dirichlet characters and of nice arithmetic functions.
For any polynomial $Q$, we say that $\chi: \mathbb{F}_q[T]\rightarrow \mathbb{C}$ is \emph{Dirichlet character} modulo $Q$ if
\begin{itemize}
\item $\chi(fg)=\chi(f)\chi(g)$ for all $f,g\in \mathbb{F}_q[T]$;
\item $\chi(1)=1$;
\item $\chi(f+gQ)=\chi(f)$ for all $f,g\in \mathbb{F}_q[T]$;
\item $\chi(f)=0$ if $\gcd(f,Q)\neq 1$.
\end{itemize}
By $\chi_0$ we denote the principal Dirichlet charater modulo $Q$. That is $\chi_0(f)=1$ for all $\gcd(f,Q)=1$. A character $\chi$ mod $Q$ is primitive if there do not exist a proper divisor $Q'$ of $Q$ and a Dirichlet character $\chi' \modu Q'$, such that $\chi(f)=\chi'(f)$ for all $f\equiv 1 \modu Q'$ and $\gcd(f,Q)=1$. A character $\chi$ is even if  $\chi(f)=\chi(cf)$ for all $f\in \mathbb{F}_q[T]$, $c\in \mathbb{F}_q^\times$.\\

\noindent Furthermore we define an arithmetic function $\alpha: \mathbb{F}_q[T]\rightarrow \mathbb{C}$ to be \emph{nice} if  
\begin{itemize}
\item $\alpha$ is even: $\alpha(f)=\alpha(cf)$ for all $f\in \mathbb{F}_q[T]$, $c\in \mathbb{F}_q^\times$.
\item $\alpha$ is multiplicative: $\alpha(fg)=\alpha(f)\alpha(g)$ for all coprime $f,g\in \mathbb{F}_q[T]$.
\item $\alpha(f)=\alpha(f^*)$ for all $f$ with $f(0)\neq 0$. Here the map $(.)^*$ is given by $f^*(T)=T^{\deg f}f\left(\frac{1}{T}\right)$.
\end{itemize} 
\noindent The starting point of proving Theorem \ref{th:TH1} will be the following lemma, first proven by Keating and Rudnick in \cite{Kea16}.
\begin{mylem}[\cite{Kea16}, Lemma 5.3]
Let $\alpha$ be a nice function. Then
\begin{equation}\label{var}
\mathrm{Var}( \mathcal{N}_\alpha)=\frac{1}{\varphi^*_{\mathrm{ev}}(T^{n-h})^2}\sum_{\substack{\chi \modu{T^{n-h}}\\ \chi\neq \chi_0 \ \mathrm{ even}}}\left|\sum_{m=0}^n\alpha(T^{n-m})\mathcal{M}(m;\alpha\chi)\right|^2.
\end{equation}
Here $\varphi^*_{\mathrm{ev}}(f)$ denotes the number of even primitive characters modulo $f\in \mathbb{F}_q[T]$ and $$\mathcal{M}(m;\alpha\chi):=\sum_{f\in \mathcal{M}_m}\alpha(f)\chi(f).$$
\end{mylem}
\noindent In particular $\varphi^*_{\mathrm{ev}}(T^{n-h})=q^{n-h-2}(q-1)$, see \cite[\S 3.3]{Kea14}.\\

\noindent Note that our function $\beta$ is even and multiplicative, as the functions $\varphi$ and $||.||$ are. Moreover $||f||=||f^*||$ for $f(0)\neq 0$. The same holds for $\varphi$, as is implied by the following lemma. We conclude that $\beta$ is nice.
\begin{mylem}\label{le:prphff} For all $f\in\mathbb{F}_q[T]$, $f\neq 0$, the following statements hold:
\begin{enumerate}
\item $||f||=\sum_{\substack{g|f,\\ g \ \mathrm{ monic}}} \varphi(g)$.
\item $\varphi = ||.|| * \mu$.
\item $$\varphi(f) = \left||f|\right| \prod_{\substack{ P|f \\ \mathrm{monic,}\\ \mathrm{irreducible}}} \left(1-\frac{1}{\left||P|\right|}\right).$$
\end{enumerate}
\end{mylem}
\begin{proof}
The proof is an exact analogue of that of an analogous lemma in $\mathbb{Z}$. It is also possiblie the prove \ref{le:prphff}.3 immediately as Rosen does, \cite[Proposition 1.7]{Ros02}. The other two statements then trivially follow.
\end{proof}
We first prove Lemma \ref{le:essumM}, which states the cancellation that makes this variance so special. Note that for an even character $\chi$, its $L$-function is given by 
$$\mathcal{L}(u,\chi)=(1-u)\prod_{j=1}^{n-h-2}(1-\alpha_j u),$$
where the $\alpha_j$ denote the inverse zeroes of this $L$-function. By the Riemann hypothesis for curves over finite fields, see for example \cite[Theorem 5.10]{Ros02}, we know that $|\alpha_j|\leq \sqrt{q}$. Furthermore for primitive $\chi$ we know that $|\alpha_j|= \sqrt{q}$, implying that we can write
$$\mathcal{L}(u,\chi)=(1-u)\det(I-uq^{\frac{1}{2}}\Theta_\chi), \ \Theta_\chi=\textrm{diag}(e^{i\theta_1},\dots,e^{i\theta_N}).$$
We say that the unitary matrix $\Theta_\chi$ of size $N:=n-h-2$ is the unitarized Frobenius matrix of $\chi$. Note that $\Theta_\chi$ is not unique, so that it is actually a conjugacy class.
\begin{mylem}\label{le:esMn}
Let $\chi$ be a primitive even character$\modu{T^{n-h}}$. Then for any $m$, we have 
$$\mathcal{M}(m;\beta\chi)=\sum_{\substack{j+k+l=m\\ 0\leq j \leq N\\ k,l \geq 0}}\lambda_j(\chi)S_l(\chi)q^{\frac{j}{2}-k-\frac{l}{2}}-\sum_{\substack{j+k+l=m-1\\ 0\leq j \leq N\\ k,l \geq 0}}\lambda_j(\chi)S_l(\chi)q^{\frac{j}{2}-k-\frac{l}{2}}.$$
Here $\lambda_j(\chi)$ is given by the coefficient of $x^j$ in the expression $\det(I_N-x\Theta_\chi)$ and $$S_l(\chi)=\tr\ \mathrm{Sym}^l \Theta_\chi.$$
\end{mylem}
\begin{proof}
We first compute the generating function of $\mathcal{M}(m;\beta\chi)$:
\begin{align*}
\sum_{m\geq 0}\mathcal{M}(m;\beta\chi)u^m&=\sum_{m\geq 0}\left(\sum_{f\in \mathcal{M}_m}\beta(f)\chi(f)\right)u^m\\
&=\sum_{f \textrm{ monic}}\beta(f)\chi(f)u^{\deg(f)}\\
&=\sum_{f \textrm{ monic}} \sum_{\substack{g|f,\\ g \textrm{ monic}}}\frac{1}{||g||}\ \mu\left(g\right)\chi(f) u^{\deg(f)}&(\textrm{lemma \ref{le:prphff}.2})\\
&=\left(\sum_{g \textrm{ monic}}\chi(g)\mu(g)\left(\frac{u}{q}\right)^{\deg(g)}\right)\left(\sum_{h \textrm{ monic}}\chi(h)u^{\deg(h)}\right) & (h=\frac{f}{g})\\
&=\frac{\mathcal{L}(u,\chi)}{\mathcal{L}(\frac{u}{q},\chi)}.
\end{align*}
We now insert $
\mathcal{L}(u,\chi)=(1-u)\det(I_N-uq^{\frac{1}{2}}\Theta_\chi)$ and note that
$$\frac{1}{\det(I_N-x\Theta_\chi)}=\sum_{l\geq 0}S_l(\chi)\ x^l.$$
Then
\begin{align*}
\sum_{m\geq 0}&\mathcal{M}(m;\beta\chi)u^m=\frac{\mathcal{L}(u,\chi)}{\mathcal{L}(\frac{u}{q},\chi)}=\frac{(1-u)\det(I_N-uq^{\frac{1}{2}}\Theta_\chi)}{(1-\frac{u}{q})\det(I_N-uq^{-\frac{1}{2}}\Theta_\chi)}\\
&=(1-u)\left(\sum_{j=0}^N\lambda_j(\chi)q^{\frac{j}{2}}u^j\right)\left(\sum_{k\geq0} q^{-k}u^k\right)\left(\sum_{l\geq 0}S_l(\chi)\ q^{-\frac{l}{2}}u^l\right)\\
&=\sum_{m\geq 0}\left(\sum_{\substack{j+k+l=m\\ 0\leq j \leq N\\ k,l \geq 0}}\lambda_j(\chi)S_l(\chi)q^{\frac{j}{2}-k-\frac{l}{2}}-\sum_{\substack{j+k+l=m-1\\ 0\leq j \leq N\\ k,l \geq 0}}\lambda_j(\chi)S_l(\chi)q^{\frac{j}{2}-k-\frac{l}{2}}\right)u^{m}.
\end{align*}
\end{proof}
\begin{mylem}\label{le:essumM}
Let $\chi$ be a primitive even character$\modu T^{n-h}$. Then
$$\sum_{m=0}^n\beta(T^{n-m})\mathcal{M}(m;\beta\chi)=\sum_{\substack{j+k+l=n\\ 0\leq j \leq N\\ k,l \geq 0}}\lambda_j(\chi)S_l(\chi)q^{\frac{j}{2}-k-\frac{l}{2}}-\sum_{\substack{j+k+l=n-1\\ 0\leq j \leq N\\ k,l \geq 0}}\lambda_j(\chi)S_l(\chi)q^{\frac{j}{2}-k-\frac{l}{2}-1}.$$
\end{mylem}
\begin{proof}
First note that applying Lemma \ref{le:esMn} several times, we have
$$\sum_{m=0}^{n-1}\mathcal{M}(m;\beta\chi)=\sum_{\substack{j+k+l=n-1\\ 0\leq j \leq N\\ k,l \geq 0}}\lambda_j(\chi)S_l(\chi)q^{\frac{j}{2}-k-\frac{l}{2}}.$$
Since $$\beta(T^{n-m})=\left\{ \begin{array}{l l} 1-\frac{1}{q} & \textrm{if $0\leq m \leq n-1$,}\\
1& \textrm{if $m=n$,}
\end{array}\right.$$
we find that
$$
\sum_{m=0}^n\beta(T^{n-m})\mathcal{M}(m;\beta\chi)=\left(1-\frac{1}{q}\right)\sum_{m=0}^{n-1}\mathcal{M}(m;\beta\chi)+\mathcal{M}(n;\beta\chi).$$
The result follows.
\end{proof}
Note that in the proof of Lemma \ref{le:esMn} and Lemma \ref{le:essumM}, we only used that $\lambda_j(\chi)$ and $S_l(\chi)$ were the coefficients of $q^{\frac{j}{2}}u^j$, $q^{\frac{l}{2}}u^l$ in the respective expressions of $\frac{\mathcal{L}(u,\chi)}{1-u}$ and $\frac{1-u}{\mathcal{L}(u,\chi)}$. Hence if we define for a non-primitive even character $\chi$ the coefficients of $q^{\frac{j}{2}}u^j$, $q^{\frac{l}{2}}u^l$ in the respective expressions of $\frac{\mathcal{L}(u,\chi)}{1-u}$ and $\frac{1-u}{\mathcal{L}(u,\chi)}$ to be $\lambda_j'(\chi)$ and $S_l'(\chi)$, then we find
\begin{mylem}\label{le:essumM2}
Let $\chi$ be a non-primitive even character$\modu T^{n-h}$. Then
$$\sum_{m=0}^n\beta(T^{n-m})\mathcal{M}(m;\beta\chi)=\sum_{\substack{j+k+l=n\\ 0\leq j \leq N\\ k,l \geq 0}}\lambda_j'(\chi)S_l'(\chi)q^{\frac{j}{2}-k-\frac{l}{2}}-\sum_{\substack{j+k+l=n-1\\ 0\leq j \leq N\\ k,l \geq 0}}\lambda_j'(\chi)S_l'(\chi)q^{\frac{j}{2}-k-\frac{l}{2}-1}.$$
\end{mylem}
\begin{proof} The proof of this is exactly analogous to that of Lemma \ref{le:esMn} and Lemma \ref{le:essumM}.
\end{proof}
We now substitute these two lemmas into equation (\ref{var}) with $\alpha=\beta$ to prove Theorem \ref{th:TH1}.

\begin{proof}[Proof of Theorem \ref{th:TH1}]
We split the sum in equation (\ref{var}) into a sum over primitive and non-primitive even characters. There are $O\left(\frac{\varphi_{\mathrm{ev}}(T^{n-h})}{q}\right)=O(q^{n-h-2})$ non-primitive even characters modulo $T^{n-h}$, see \cite[\S 3.3]{Kea14}. The largest power of $q$ in the first sum of Lemma \ref{le:essumM2} is given by $\frac{n-2h-4}{2}$. It is only attained when $j=N=n-h-2, k=0, l=n-N=h+2$. In the second sum of Lemma \ref{le:essumM2} the largest power of $q$ is given by $\frac{n-2h-5}{2}$, which is attained when $j=N, k=0, l=n-N-1$. It follows that the sum over all non-primitive characters is estimated by $O\left(q^{-2(n-h-1)}\cdot q^{n-h-2}\cdot q^{n-2h-4}\right)=O\left(q^{-h-4}\right)$. Applying the same arguments for the largest $q$-powers in lemma \ref{le:essumM}, we find
\begin{align*}
\mathrm{Var}( \mathcal{N}_\beta)&=\frac{1}{\varphi^*_{\mathrm{ev}}(T^{n-h})^2}\sum_{\substack{\chi \modu{T^{n-h}}\\ \chi\neq \chi_0 \ \textrm{even, primitive}}}\left|\sum_{m=0}^n\alpha(T^{n-m})\mathcal{M}(m;\alpha\chi)\right|^2+O\left(q^{-h-4}\right)\\
&=\frac{1}{\varphi^*_{\mathrm{ev}}(T^{n-h})^2}\left(\sum_{\substack{\chi \modu{T^{n-h}}\\ \chi\neq \chi_0 \ \textrm{even, primitive}}} |\lambda_N(\chi)|^2 |S_{h+2}(\chi)|^2\right) q^{n-2h-4} + O\left(q^{-h-\frac{7}{2}}\right)\\
&=\frac{\sum_\chi^* |\lambda_N(\chi)|^2 |S_{h+2}(\chi)|^2}{\varphi^*_{\mathrm{ev}}(T^{n-h})} q^{-h-3} + O\left(q^{-h-\frac{7}{2}}\right).
\end{align*}
Note that $|\lambda_N(\chi)|^2=1$ for all primitive $\chi$. Katz's equidistribution theorem for primitive even characters modulo $T^{m}$, \cite{Kat13}, states that, if $m\geq 5$, the Frobenius matrices of these characters become equidistributed in $PU(m-2)$ in the limit $q\rightarrow \infty$. This theorem enables us to replace the average over primitive even characters modulo $T^{n-h}$ by a matrix integral over $PU(n-h-2)$.
Finally note that we can replace the matrix integral over the projective group $PU(n-h-2)$ by an integral over the unitary group $U(n-h-2)$, since the function we average over is invariant under scalar multiplication. Since the symmetric $m$-th power is an irreducible representation for any $m$, see for example \cite[Lecture 6]{Ful91}, we conclude that
$$\int_{U(n-h-2)}\left|\tr \ \sym^{h+2}(U)\right|^2dU=1.$$ Hence, as $q\rightarrow \infty$, 
$$\mathrm{Var}( \mathcal{N}_{\beta})\sim q^{-h-3}.$$
\end{proof}


\begin{thebibliography}{10}

\bibitem{Apo76}
T.~M. Apostol.
\newblock {\em Introduction to analytic number theory}.
\newblock Springer-Verlag, New York-Heidelberg, 1976.
\newblock Undergraduate Texts in Mathematics.

\bibitem{Cho32}
S.~Chowla.
\newblock Contributions to the analytic theory of numbers.
\newblock {\em Math. Z.}, 35(1):279--299, 1932.

\bibitem{Sha55}
P.~Erd\"os and H.~N. Shapiro.
\newblock The existence of a distribution function for an error term related to
  the {E}uler function.
\newblock {\em Canad. J. Math.}, 7:63--75, 1955.

\bibitem{Ful91}
W.~Fulton and J.~Harris.
\newblock {\em Representation theory}, volume 129 of {\em Graduate Texts in
  Mathematics}.
\newblock Springer-Verlag, New York, 1991.
\newblock A first course, Readings in Mathematics.

\bibitem{Gol87}
D.~A. Goldston and H.~L. Montgomery.
\newblock Pair correlation of zeros and primes in short intervals.
\newblock In {\em Analytic number theory and {D}iophantine problems
  ({S}tillwater, {OK}, 1984)}, volume~70 of {\em Progr. Math.}, pages 183--203.
  Birkh\"auser Boston, Boston, MA, 1987.

\bibitem{Kat13}
N.~M. Katz.
\newblock Witt vectors and a question of {K}eating and {R}udnick.
\newblock {\em Int. Math. Res. Not. IMRN}, (16):3613--3638, 2013.

\bibitem{Rud16}
J.~Keating, B.~Rodgers, E.~Roditty-Gershon, and Z.~Rudnick.
\newblock Sums of divisor functions in {$F_q[t]$} and matrix integrals.
\newblock \url{https://arxiv.org/abs/1504.07804}.

\bibitem{Kea16}
J.~Keating and Z.~Rudnick.
\newblock Squarefree polynomials and {M}\"obius values in short intervals and
  arithmetic progressions.
\newblock {\em Algebra Number Theory}, 10(2):375--420, 2016.

\bibitem{Kea14}
J.~P. Keating and Z.~Rudnick.
\newblock The variance of the number of prime polynomials in short intervals
  and in residue classes.
\newblock {\em Int. Math. Res. Not. IMRN}, (1):259--288, 2014.

\bibitem{Mer77}
F.~Mertens.
\newblock Ueber einige asymptotische {G}esetze der {Z}ahlentheorie.
\newblock {\em J. Reine Angew. Math.}, 77:289--338, 1874.

\bibitem{Mon87}
H.~L. Montgomery.
\newblock Fluctuations in the mean of {E}uler's phi function.
\newblock {\em Proc. Indian Acad. Sci. Math. Sci.}, 97(1-3):239--245 (1988),
  1987.

\bibitem{Ros02}
M.~Rosen.
\newblock {\em Number theory in function fields}, volume 210 of {\em Graduate
  Texts in Mathematics}.
\newblock Springer-Verlag, New York, 2002.

\bibitem{Wal63}
A.~Walfisz.
\newblock {\em Weylsche {E}xponentialsummen in der neueren {Z}ahlentheorie}.
\newblock Mathematische Forschungsberichte, XV. VEB Deutscher Verlag der
  Wissenschaften, Berlin, 1963.

\end{thebibliography}
\end{document}